\documentclass{article}
\usepackage[margin=1in]{geometry}

\AtEndDocument{\bigskip{\footnotesize%
  \textsc{Department of Applied Mathematics, Delhi Technological University, Delhi–110042, India} \par
  \textit{E-mail address:} \texttt{spkumar@dtu.ac.in} \par
  \addvspace{\medskipamount}
  \textsc{Department of Applied Mathematics, Delhi Technological University, Delhi–110042, India} \par  
  \textit{E-mail address:} \texttt{suryagiri456@gmail.com} \par
}}
\usepackage{graphicx}
\usepackage{hyperref}
\usepackage[caption = false]{subfig}
\usepackage{amsthm}
\usepackage{amssymb}
\usepackage{mathrsfs}
\usepackage{mathtools}
\usepackage{enumerate}
\usepackage{amssymb}
\usepackage{wrapfig}
\usepackage{float}
\allowdisplaybreaks

\usepackage{geometry}
\numberwithin{equation}{section}
\DeclareMathOperator{\RE}{Re}

\theoremstyle{plain}
\usepackage{lipsum}

\newtheorem{theorem}{Theorem}[section]
\newtheorem{corollary}[theorem]{Corollary}
\newtheorem{example}[theorem]{Example}
\newtheorem{lemma}{Lemma}[section]
\newtheorem{proposition}[theorem]{Proposition}
\theoremstyle{definition}

\theoremstyle{remark}
\newtheorem{remark}{Remark}[section]

\makeatother

\usepackage{authblk}
\begin{document}
\title{Radius and Convolution problems of  analytic functions involving Semigroup Generators}
\author{Surya Giri and S. Sivaprasad Kumar}


\date{}


	

\maketitle	
	
\begin{abstract}
 We establish the membership criteria in terms  of Hadamard product for a normalized analytic function to be in the class of infinitesimal generators. Furthermore, the embedding of various subclasses of normalized univalent functions in the class of infinitesimal generators and  the radii  problems for this class  are studied. The derived results  generalize the already known results.
\end{abstract}

\vspace{0.5cm}
\maketitle

\section{Introduction}
   Let $\mathcal{H}(\mathbb{D},\mathbb{C})$ be the class of holomorphic functions in the unit disk $\mathbb{D}$, $\mathcal{A}$ be a subclass of $\mathcal{H}(\mathbb{D},\mathbb{C})$ consisting of functions $f$  satisfying $f(0)=f'(0)-1=0$. Let $\mathcal{S} \subset \mathcal{A}$ be the class of univalent functions. By $\mathcal{H}(\mathbb{D})$, we mean the class of holomorphic self mappings of the unit disc $\mathbb{D}$. Let us now recall some definitions related to our work.

 A family $\{ u(t,z) \}_{t\geq 0} \subset \mathcal{H}(\mathbb{D})$ of holomorphic self mappings of $\mathbb{D}$ is called a one-parameter continuous semi-group or semiflow on $\mathbb{D}$, if
\begin{enumerate}
\item  $u(t+s,z) = u(t,z) \circ u(s,z)$ for all $s,t\geq 0$;
  \item $\lim_{t\rightarrow s} u(t, z) = u(s, z);$
  \item $\lim_{t\rightarrow 0^+}u(t,z)=z$ for each $z\in \mathbb{D}$ and the limit is taken with  respect to the topology of uniform convergence on compact sets in $\mathbb{D}$.
\end{enumerate}

    According to Berkson and Porta \cite{BP}, every one-parameter continuous semigroup is differentiable with respect to the parameter and further if
  $$    \lim_{t\rightarrow 0^{+}}  \frac{z - u(t,z)}{t}= f(z) ,  $$
   which is a holomorphic function on $\mathbb{D}$, then the Cauchy problem
\begin{equation}\label{CP}
     \frac{\partial u(t,z)}{\partial t} + f(u(t,z))=0,\;\;\; u(0,z)=z
\end{equation}   has a unique solution $u(t,z)\in \mathcal{H}(\mathbb{D})$, $t \geq 0$.
    The function $f$ is called the holomorphic or infinitesimal generator of one parameter continuous semigroup $\{ u(t,z)\}_{t\geq 0}$.
   We denote the class of all holomorphic generators by $\mathcal{G}$.

   It is worth noting that each element of continuous semi-group generated by $f\in \mathcal{G}$ is a  univalent function on $\mathbb{D}$ however $f$ need not be univalent \cite{Ellin2}.
   In the past, several analytic criteria have been established to determine whether a function is an infinitesimal generator by many authors \cite{BP,Shoi,FB}.
    Berkson and Porta \cite{BP} showed that:
\begin{theorem}\label{thm}
    The following assertions are equivalent:
\begin{enumerate}
  \item[$(a)$] $f\in \mathcal{G}$;
  \item[$(b)$] $f(z)=(z - \sigma)(1- z \bar{\sigma})p(z)$ with some $\sigma \in \overline{\mathbb{D}}$ and $p\in \mathcal{H}$, $\RE(p(z))\geq 0$, $z\in \mathbb{D}$;
\end{enumerate}
\end{theorem}
  The point $\sigma \in \mathbb{\overline{D}}:=\{ z \in \mathbb{C}: \lvert z \rvert \leq 1\}$ stated in the above theorem, is called the Denjoy--Wolff point of the semigroup generated by $f$. According to the Denjoy-Wolff theorem \cite{Ellin2,Shoi,DWT} for continuous semigroup, if the semigroup generated by $f$ is not an elliptic automorphism of $\mathbb{D}$ and is not an identity map for at least one $t\in [0,\infty)$, then there is a unique point $\sigma \in \mathbb{\overline{D}}$ such that $\lim_{t\rightarrow \infty} u(t,z)= \sigma$ uniformly.
   If $\sigma \in \partial \mathbb{D} $, it is known as the sink point or Wolff's point. Let the class of infinitesimal generators with Denjoy-Wolff point $\sigma$  be denoted by $\mathcal{G}[\sigma].$

  For $\sigma=0$, we obtain the following subclass
  $$ \mathcal{G}[0]= \{f\in  \mathcal{G} : f(z)= z p(z), \;\; \RE p(z) > 0 \}.$$
 Although the semigroup $\{u(t,\cdot)\}_{t\geq 0}$ generated by $f\in \mathcal{G}[0]$ is real analytic with respect to its parameter, it does not always allow for an analytic extension to a domain in $\mathbb{C}$ (see \cite{Ellin2}). Elin et al. \cite{Elin3} proved the following:
\begin{proposition}
    The semigroup generated by $f \in  \mathcal{G}[0]$ can be analytically extended to the sector $\{t : \lvert \arg t \rvert < \pi \alpha/2 \}$ if and only if  $\lvert\arg (f(z) /z)  \rvert < \pi (1 - \alpha)/2 $, for all  $z \in  \mathbb{D}$.
\end{proposition}

  Various parameterizations of the class $\mathcal{G}[0]$ are considered with $f\in \mathcal{A}$, which is known as filtration theory and introduced in \cite{DShoi}. Later on, Bracci et al. \cite{FB1} developed this concept in more detail. Among other things, they considered the class
   $$ \mathcal{G}_0  = \mathcal{G}[0]\cap \mathcal{A}$$
   and studied the asymptotic behaviour of the semigroup generated by $f\in \mathcal{G}_{0}$, as well as the analytic extension of the semigroup in terms of its parameter to the domain in $\mathbb{C}$. They showed that
\begin{proposition}\label{prp1}
     The semigroup $\{ u(t, \cdot)\}_{t \geq 0}$ generated by $f(z) = z p(z)$ has a uniform exponential rate of convergence:
    $\lvert u(t,z) \rvert \leq \lvert z \rvert e^{-t k}$ if and only if $\RE p(z)\geq k >0$ for all $z \in \mathbb{D}$.
\end{proposition}
   For $\beta \in [0,1]$, a filtration of $\mathcal{G}_0$ is
\begin{equation}\label{1eq}
   \mathcal{A}_\beta = \bigg\{ f\in \mathcal{A}: \RE \bigg(\beta \frac{f(z)}{z} +(1- \beta) f'(z)\bigg) > 0 \bigg\}
\end{equation}
   as $\mathcal{A}_{\beta_1} \subsetneq \mathcal{A}_{\beta_2} \subsetneq \mathcal{G}_0$ for $0 \leq \beta_1 <\beta_2 <1$ (see \cite{FB1}). When $\beta =1$, $\mathcal{A}_1 = \mathcal{G}_0$ and for $\beta =0$, the class $\mathcal{A}_\beta$ reduces to the class
   $$\mathcal{R} = \{ f\in \mathcal{A}: \RE f'(z) > 0 \},$$
   where $\mathcal{R}$ is the class of bounded turning functions.  It follows from \cite[Lemma 4]{RamSingh} that $\mathcal{R}\subset \mathcal{G}_0$.

   Now, let us recall that a univalent function $f\in \mathcal{H}(\mathbb{D},\mathbb{C})$ with $f(0)=0$ is said to be starlike in $\mathbb{D}$ if  $e^{-t} f(z)$ lies in the image domain of $f$, for any $t \geq 0$. We note that each starlike function $f$ on $\mathbb{D}$ yields a family $\{u(t,\cdot)\}_{t\geq 0} \subset \mathcal{H}(\mathbb{D})$ defined by
   $$ u(t,z) = f^{-1}(e^{-t}f(z)).$$
   It is an easy exercise to see that this family is a one parameter continuous semigroup. Differentiating this at $t=0^+$, we obtain that
   $$ f(z) = f'(z) g(z),$$
   where $g$ is the generator of the semigroup $\{ u(t,\cdot)\}_{t\geq 0}$ (see \cite{Ellin2}). This statement is also true in general case.
\begin{theorem}\cite{Ellin2}
   Let $f \in \mathcal{H}(\mathbb{D},\mathbb{C})$. Then $f$ is a starlike function if and only if
    $ f(z) = f'(z) g(z),  $
    where $g \in \mathcal{G}$.
\end{theorem}
   Various subclasses of starlike functions are unified by Ma and Minda \cite{MaMinda}. They defined
        $$ \mathcal{S}^*(\varphi) = \bigg\{ f\in \mathcal{S}: \frac{z f'(z)}{f(z)}\prec \varphi(z) \bigg\},$$
    where $\varphi$ is analytic univalent function such that the image domain of $\mathbb{D}$ under $\varphi$ is starlike with respect to $\varphi(0)=1$, lies in the right half plane and symmetric about the real axis. For $\varphi(z)=(1+Az)/(1+Bz)$, $-1\leq B < A \leq 1$, the class $\mathcal{S}^*(\varphi)$ reduces to the class of Janowski starlike functions \cite{janowski}. For $A = 1-2 \alpha$ and $B=-1$, we obtain the class of starlike functions of order $\alpha$,  $\mathcal{S}^*(\alpha)$  ($0\leq \alpha <1).$  Subclasses of $\mathcal{S}^*$ are studied for various choices of $\varphi$. For these subclasses, we refer \cite{SG,kamal}.

    Kumar and Gangania \cite{kamal2} introduced and studied the class
    $$ \mathcal{F}(\Psi) = \bigg\{ f \in \mathcal{A} : \left(\frac{z f'(z)}{f(z)}-1 \right) \prec \Psi(z),\;\; z\in \mathbb{D} \bigg\},$$
    where $\Psi$ is univalent function and $\Psi(0)=0$. It should be noted that this class also contains non-univalent functions. For $\Psi(z)=z/(1-\alpha z^2)$, $\alpha \in [0,1)$, the class $\mathcal{F}(\Psi)$ reduces to the class
    $$ \mathcal{BS}(\alpha) =\left\{ f \in \mathcal{A} : \left(\frac{z f'(z)}{f(z)}-1 \right) \prec \frac{z}{1 - \alpha z^2},\;\; z\in \mathbb{D} \right\}$$
    introduced by Karger et al. \cite{karger2}. The geometric properties of $f\in \mathcal{BS}(\alpha)$ including radii problems for starlike functions of order $\alpha$ are studied in \cite{karger}.
   Another interesting class is
\begin{equation}\label{classU}
     \mathcal{U}(\lambda)=\bigg\{ f\in \mathcal{A}: \bigg\lvert f'(z) \left( \frac{z}{f(z)}\right)^2 -1 \bigg\rvert < \lambda, \;\; \lambda \in (0,1] \bigg\},
\end{equation}
    introduced by Obradovi\'{c} and Ponnusamy \cite{Obra}. It is well known that $u (\lambda) \subset \mathcal{S}$ for $\lambda \in [0,1].$ For more works on this class, see \cite{obra2,obra3} and the references cited therein.

    Functions $f(z)= z/(1-z + z^2)$ and $g(z)=z(1+z)/(1-z)$ reveals  that neither $\mathcal{S}^* \subset\mathcal{G}[0]$ nor  $\mathcal{G}[0] \subset \mathcal{S}^*.$
    Here, the radius problem arise, in which, we find the largest $r\in (0,1)$ such that $f(rz)/r \in \mathcal{S}^*$ whenever $f\in \mathcal{G}[0]$.
    Elin et al. \cite{Elin} solved this problem for the class $\mathcal{A}_\beta$, which immediately provides the radius of starlikeness for the class $\mathcal{G}[0]$, when $\beta =1$. They proved that the radius of starlikeness is $r= 2 - \sqrt{2}$ for the class $\mathcal{G}[0]$. Further, they obtained sufficient condition for the class $\mathcal{A}_\beta$ using the technique of differential subordination and also see the inclusion of various subclasses of starlike functions in $\mathcal{A}_\beta$.

    Extending this work, we find the radius for the class $\mathcal{S}^*(\varphi)$ in section \ref{sec2}. We also solve radii problems for the classes $\mathcal{F}(\Psi)$, $\mathcal{U}(\lambda)$ and find the uniform exponential rate of convergence of semigroup generated by the members of these classes. Further, we also prove that the convolution of $f\in \mathcal{A}_\beta$ with $g\in \mathcal{K}$ is again in $\mathcal{A}_\beta$ and the class $\mathcal{A}_\beta$ is preserved under some integral operators, where  $\mathcal{K}\subset \mathcal{A}$ is the class of convex functions.  The Hadamard product or convolution of two functions $f(z) = \sum_{n=0}^\infty a_n z^n$ and $g(z)= \sum_{n=0}^\infty b_n z^n$ is defined as $(f*g)(z) =\sum_{n=0}^\infty a_n b_n z^n$. The following lemmas help us in proving our results.
\begin{lemma}\label{lemma2}\cite{ram}
   If $g(z)$ is analytic in $\mathbb{D}$, $g(0)=1$ and $\RE(g(z)/z)>1/2$, $z\in \mathbb{D}$, then for any function $f$, analytic in $\mathbb{D}$, the function $g* f$ takes values in the convex hull of the image of $\mathbb{D}$ under $f$.
\end{lemma}
\begin{lemma}\label{lemma1}\cite{rusch}
   If $g \in \mathcal{K}$ and $h \in \mathcal{S}^*$, then for each function $F$, analytic in $\mathbb{D}$, the image of $\mathbb{D}$ under $(g * F h)/(g * h)$ is a subset of the convex hull of $F(\mathbb{D})$.
\end{lemma}
\section{Convolution Properties}\label{sec1}
    We start with the following membership criteria for the class $\mathcal{A}_\beta$.
\begin{theorem}
   Let $f\in \mathcal{A}$, then $f\in \mathcal{A}_\beta$ if and only if
   $$ f(z) * z \left(\frac{z( 1 - z \beta)}{(1  - z)^2}  - \frac{1+\zeta}{1-\zeta} \right) \neq 0, \quad \lvert \zeta \rvert =1.$$
\end{theorem}
\begin{proof}
       From (\ref{1eq}), $f(z) = z + \sum_{n=2}^\infty a_n z^n \in \mathcal{A}_\beta$ if and only if
       $$  \beta  \frac{f(z)}{z}  + ( 1 - \beta ) f'(z) \prec \frac{1+z}{1-z}, $$
       which ensure the existence of a Schwarz function $\omega$ such that
        $$    \beta  \frac{f(z)}{z}  + ( 1 - \beta ) f'(z) = \frac{1+\omega(z)}{1- \omega(z)}. $$
        By the property of subordination, we have
\begin{equation}\label{eqn5}
      \beta  \frac{f(z)}{z}  + ( 1 - \beta ) f'(z) \neq \frac{1+ \zeta}{1- \zeta}, \quad (z\in \mathbb{D})
\end{equation}
     where  $\lvert \zeta \rvert =1$. Using the following basic convolution properties
      $$z = f(z) *z ,\;\; f(z)=f(z) * \frac{z}{(1-z)} \;\; \text{and}  \;\; zf'(z)=f(z)* \frac{z}{(1-z)^2} $$
      in (\ref{eqn5}), we obtain
      $$ f(z) * \left(  \frac{\beta z}{1-z}  +  \frac{(1 - \beta ) z}{(1-z)^2} - \frac{z(1+\zeta)}{1-\zeta} \right) \neq 0,$$
      which completes the result.
\end{proof}
\begin{theorem}\label{thm2}
     Let $f(z)= z+ \sum_{n=2}^\infty a_n z^n \in  \mathcal{A}$. If
\begin{equation}\label{eqnf}
       \left\lvert \sum_{n=1}^\infty \bigg( (n(\beta -1)- \beta) a_n + (n (1-\beta) +1) a_{n+1} \bigg) \right\rvert  \leq 1,
\end{equation}
      then $f \in \mathcal{A}_\beta$ and the inequality is sharp.
\end{theorem}
\begin{proof}
   To prove $f\in \mathcal{A}_\beta$, it is sufficient to show that
\begin{equation*}
      \left\lvert  \bigg( (1-\beta)  f'(z) + \beta \frac{f(z)}{z} \bigg) (1-z)-1 \right\rvert <1.
\end{equation*}
    Now,
\begin{align*}
    \bigg((1-\beta)  f'(z) + &\beta \frac{f(z)}{z} \bigg) (1-z)-1   \\
     &=  \sum_{n=1}^\infty \bigg( -n a_n +(n+1) a_{n+1} + (n-1) a_n \beta -n a_{n+1} \beta \bigg) z^n .
\end{align*}
   Since (\ref{eqnf}) holds, therefore the last expression is bounded above by 1 and the result follows at once.
  Now,  for $\beta \in [0,1)$,  sharpness of the result follows by considering the functions $f$, defined by $f: \mathbb{D} \rightarrow \mathbb{C}$  such that
\begin{equation}\label{eqn0}
    f(z)=  z \left( -1 + 2 \left( {}_2 F_1 \left[1, \frac{1}{1-\beta}, \frac{2- \beta}{1-\beta},z \right] \right) \right)= z + \sum_{n=2}^\infty \frac{2}{n -(n-1)\beta} z^n,
\end{equation}
   where $F$ is hypergeometric function. Inequality in (\ref{eqnf}) becomes  equality in case of $f$ defined in (\ref{eqn0}) whereas, when $\beta =1$, 
   extremal function is $f(z)=z (1+z)/(1-z)$.
\end{proof}
    If we put $\beta=1$ in Theorem \ref{thm2}, we obtain a sufficient condition for $\mathcal{G}_0.$
\begin{corollary}
     Let $f(z)=z+ \sum_{n=2}^\infty a_n z^n \in \mathcal{A}$. If $\left\lvert \sum_{n=1}^\infty (  a_{n+1} - a_n ) \right\rvert  \leq 1,$
     then $f \in \mathcal{G}_0$ and the inequality is sharp.
\end{corollary}
\begin{example}
    Polynomial $f(z)= z + a_n z^n \in \mathcal{G}_0$, ($n \geq 2$) whenever  $\lvert a_n -1 \rvert \leq 1$.
\end{example}
\begin{theorem}\label{thm4}
    If $f\in \mathcal{A}_\beta$ and $g\in \mathcal{A}$ such that $\RE(g(z)/z)>1/2$, then $f * g \in \mathcal{A}_\beta$.
\end{theorem}
\begin{proof}
      For $F(z)= (f * g)(z)$, we have
   $$ z F'(z) = z f'(z) * g(z).$$
   Thus,
\begin{align*}
    \RE \bigg( \frac{\beta F(z)+ (1 - \beta) z F'(z)}{z}\bigg) &=  \RE  \left(  \frac{\beta ( f(z) * g(z)) + (1 - \beta)z f'(z) * g(z)}{z}  \right),\\
                                                               &= \RE  \left( \left( \frac{\beta f(z)  + (1 - \beta)z f'(z) }{z} \right) * g(z) \right).
\end{align*}
   Since $\RE \left( \left( \beta f(z)  + ( 1 - \beta ) z f'(z)\right) /z \right) > 0$ and by the hypothesis $\RE(g(z)/z)>1/2$, therefore by Lemma \ref{lemma2},
   $$\RE \left(\frac{\beta F(z)+ (1 - \beta) z F'(z)}{z} \right) > 0. $$
   Consequently, $F(z)=  f* g \in \mathcal{A}_\beta$, which completes the proof.
\end{proof}
\begin{theorem}\label{thm3}
    If $f\in \mathcal{A}_\beta$ and $g \in \mathcal{K}$, then $f * g \in \mathcal{A}_\beta$.
\end{theorem}
\begin{proof}
    Let $f\in \mathcal{A}_\beta$, then we have $\RE F(z)>0$, where
    $$F(z):= \frac{\beta f(z)  + (1 - \beta)z f'(z) }{z}. $$
    For $g \in \mathcal{K}$, we have
      $$ \frac{g *z  F }{g * z} = \left( \frac{\beta (g* f)  + (1 - \beta)z (g * f)' }{z} \right)  $$
    which together with Lemma \ref{lemma2} yields that $\RE\left( \left( {\beta (g* f)  + (1 - \beta)z (g * f)' }/{z}  \right)\right)>0$, hence $f * g \in \mathcal{A}_\beta$.
\end{proof}
\begin{remark}
     By Marx Stroh\"{a}cker result \cite{good}, we have $\RE (g(z)/z)>1/2$ whenever $g \in \mathcal{K}$. Thus,  Theorem \ref{thm4} also follows from Theorem \ref{thm3}.
\end{remark}
    If we take
     $$g(z)= \sum_{n=1}^\infty \left(\frac{1+\gamma}{n+\gamma }\right) z^n, \;\;\; \RE{\gamma}\geq -\frac{1}{2}, $$
     which is a convex function in $\mathbb{D}$ (see \cite{rusch2}), in Theorem \ref{thm3}, the following result follows:
\begin{corollary}
   If $f\in \mathcal{A}_\beta$, then so is
   $$ \frac{1+\gamma}{z^\gamma} \int_0^z t^{\gamma-1} f(t) dt , \;\;\;  \RE{\gamma}\geq -\frac{1}{2}.$$
\end{corollary}
   For $\gamma =0$ and $\gamma=1$, function $g$ reduces to $-\log{(1-z)}$ and $- 2(z + \log(1-z))/z$ respectively and from Theorem \ref{thm3}, we obtain the following:
\begin{corollary}
\begin{enumerate}[(i)]
  \item If $f\in  \mathcal{A}_\beta$, then
   $$ \int_0^z \frac{f(t)-f(0)}{t} dt = \int_0^z \frac{f(t)}{t} dt   \in \mathcal{A}_\beta.$$
  \item If $f\in \mathcal{A}_\beta$, then
   $$ \frac{2}{z}\int_0^z f(t) dt \in \mathcal{A}_\beta.$$
\end{enumerate}
\end{corollary}
    Consider the function
    $$ g(z) = \frac{1}{1-x}\log \left( \frac{1 - x z}{1 - z} \right), \;\; \lvert x \rvert \leq 1, \;\; x \neq 1.$$
     Since $g(z)$ is a convex function, for $f\in \mathcal{A}_\beta$, Theorem \ref{thm3} yields:
\begin{corollary} If $f\in \mathcal{A}_\beta$, then
     $$ \int_0^z \frac{f(t)- f(x t)}{t - x t} dt \in \mathcal{A}_\beta.$$
\end{corollary}
\section{Inclusion and Radius Problems}\label{sec2}
\begin{theorem}\label{tpsi1}
    If $\Psi$ is convex, then $\mathcal{F}(\Psi)\subset \mathcal{G}_0$ whenever
\begin{equation}\label{psi3}
    \RE \left( \exp\int_0^z \frac{\Psi(t)}{t} \right) \geq \delta >0.
\end{equation}
   Moreover, semigroup  generated by $f\in \mathcal{F}(\Psi)$ satisfies $\lvert u(t,\cdot)\rvert \leq e^{-t \delta}\lvert z \rvert .$
\end{theorem}
\begin{proof}
    Let $f\in \mathcal{F}(\Psi)$, then
\begin{equation}\label{psi}
    \frac{z f'(z)}{f(z)}-1 \prec \Psi(z).
\end{equation}
    It can be easily seen that
    $$ s(z):= \log\left( \frac{1}{1-z}\right) = z+ \frac{z^2}{2} + \frac{z^3}{3} +\cdots $$
    is convex univalent function.
    Using the result \cite{Rusch}, which states that: for any convex univalent functions $F$ and $G$  in $\mathbb{D}$, if $f\prec F$ and $g \prec G$, then $f * g \prec F * G $, in (\ref{psi}), we get
\begin{equation}\label{psi2}
    \left( \frac{z f'(z)}{f(z)}-1 \right) * \log\left( \frac{1}{1-z}\right)   \prec \Psi(z) * \log\left( \frac{1}{1-z}\right) .
\end{equation}
    Now, applying the following convolution properties
\begin{align*}
    \left( \frac{z f'(z)}{f(z)}-1 \right) * \log\left( \frac{1}{1-z}\right) & =\int_0^z \frac{1}{t}\left( \frac{t f'(t)}{f(t)}-1 \right) dt, \\
    \;\; \Psi(z) * \log\left( \frac{1}{1-z}\right) &=\int_0^z \frac{\Psi(t)}{t}dt,
\end{align*}
   in (\ref{psi2}), we obtain
   $$\int_0^z \frac{1}{t}\left( \frac{t f'(t)}{f(t)}-1 \right) dt \prec \int_0^z \frac{\Psi(t)}{t}dt.$$
   Consequently
   $$  \frac{f(z)}{z} = \exp \int_0^z \frac{1}{t}\left( \frac{t f'(t)}{f(t)}-1 \right) dt \prec \exp\int_0^z \frac{\Psi(t)}{t}dt.$$
   By subordination principle for $\lvert z \rvert \leq r <1$ (see \cite{good}), we have
   $$ \RE \left(\frac{f(z)}{z} \right) \geq \RE \left( \exp\int_0^z \frac{\Psi(t)}{t}dt \right).$$
   Thus, whenever (\ref{psi3}) holds, we have $\mathcal{F}(\Psi)\subset \mathcal{G}_0$ and the result follows at once from Proposition \ref{prp1}.
\end{proof}
\begin{example}
   Let us take $\Psi(z)= -2 z /(1-z^2)$, then it can be easily seen that $\Psi(z)$ is analytic, univalent and starlike function with respect to $0$ in $\mathbb{D}$ and $\Psi(z)$ is convex in the disk of radius $ \lvert z \rvert \leq r_0 \approx 0.414214$, where $r_0$ is the root of the equation
   $$r^4 -6 r^2+1 =0. $$
    For this $\Psi$, we can consider
   $$ \mathcal{F}_1(\Psi)= \left\{f \in \mathcal{A}: \frac{z f'(z)}{f(z)}-1 \prec \frac{-2 z}{1- z^2}\right\}$$
  and a simple calculation yields that
   $$  \RE \left( \exp\int_0^z \frac{\Psi (t)}{t} \right) = \RE \left(\frac{1- z}{1+ z} \right)>0 .$$
    Therefore, by Theorem \ref{tpsi1}, $\mathcal{F}_1(\Psi) \subset \mathcal{G}_0$ in $\lvert z \rvert \leq r_0$.
\end{example}
    For $\Psi(z)= z/(1- \alpha z^2)$, the class $\mathcal{F}(\Psi)$ reduces to the class $\mathcal{BS}(\alpha)$. By Theorem \ref{tpsi1}, we obtain the following:
\begin{corollary}
    For $ 0 < \alpha \leq 3 - 2\sqrt{2}$, $\mathcal{BS}(\alpha) \subset \mathcal{G}_0$ and semigroup generated by $f\in \mathcal{BS}(\alpha)$ satisfies
    $$  \lvert u(t,z) \rvert \leq e^{-t \left( \frac{1 - \sqrt{\alpha} }{1+ \sqrt{a} }  \right)^{\frac{1}{2 \sqrt{\alpha}}}} \lvert z \rvert.  $$
\end{corollary}
\begin{proof}
  Since $\Psi(z)=z/(1-\alpha z^2)$ is convex for $0 < \alpha \leq 3 - 2\sqrt{2}$ \cite[Lemma 3.1]{piejko} and
\begin{align*}
    \RE \left( \exp\int_0^z \frac{\Psi(t)}{t} \right) &= \RE \left( \frac{1+ \sqrt{\alpha} z}{1- \sqrt{\alpha} z}  \right)^{\frac{1}{2 \sqrt{\alpha}}} \\
                                                      &=: \RE g(z),
\end{align*}
   where $g(z)=((1+ \sqrt{\alpha} z)/(1 - \sqrt{\alpha} z))^{1/(2 \sqrt{\alpha})}$. By \cite[Theorem 2.4]{karger}, function $g$ is convex univalent in $\mathbb{D}$ and $g(z)$ is real for real $z$, therefore it maps the unit disk onto a convex set symmetric with respect to the real axis lying between $g(-1)$ to $g(1).$
   Thus $\RE g(z) \geq ((1 - \sqrt{\alpha} )/(1 + \sqrt{\alpha} ))^{1/(2 \sqrt{\alpha})}>0 $ for $ 0 < \alpha \leq 3 - 2\sqrt{2}$, which together with Theorem \ref{tpsi1} and Proposition \ref{prp1} establsih the result.
\end{proof}
\begin{theorem}\label{thmn}
  The  inclusion $\mathcal{S}^*(\varphi) \subset \mathcal{G}_0$ holds, whenever
\begin{equation}\label{frst}
    \RE \left\{ \exp \int_0^{z}  \frac{\varphi(t)-1}{t} dt \right\}\geq \gamma >0.
\end{equation}
     Moreover, $f\in \mathcal{S}^*(\varphi)$ generates the semigroup $\left\{ u(t,\cdot )\right\}_{t\geq 0}$, which satisfies
    $$ \lvert u(t,z) \rvert \leq e^{-t \gamma} \lvert z \rvert.$$
\end{theorem}
\begin{proof}
    Let $f\in \mathcal{S}^*(\varphi)$, then from \cite[Theorem 1]{MaMinda}, we have
    $$ \frac{f(z)}{z} \prec \frac{\Tilde{f}(z)}{z},$$
    where $$\Tilde{f}(z)=z \exp \left( \int_0^z \frac{\varphi(t)-1}{t} dt \right).$$
    By subordination principle for $\lvert z \rvert \leq r <1,$
\begin{equation}\label{stro}
    \RE \bigg(\frac{f(z)}{z} \bigg) \geq \RE\bigg( {\frac{\Tilde{f}(z)}{z}} \bigg)=\RE \bigg( \exp \bigg( \int_0^z \frac{\varphi(t)-1}{t} dt \bigg)\bigg)> \gamma.
\end{equation}
     Hence, $\mathcal{S}^*(\varphi)\subset \mathcal{G}_0$ and the result follows at once with Proposition \ref{prp1}.
\end{proof}
\begin{remark}
   For $\varphi(z)=1/(1-z)$, the class $\mathcal{S}^*(\varphi)$ reduces to the class $\mathcal{S}^*(1/2)$, which means  $\RE ({z f'(z)}/{f(z)}) > {1}/{2}$ and from (\ref{stro}), we have $\RE(f(z)/z)>1/2$ $(z \in \mathbb{D})$ for $f\in \mathcal{S}^*(1/2)$, proved by Marx-Stroh\"{a}cker \cite{Strohack}.
\end{remark}
\begin{corollary}\label{crl1}
    If $-1\leq B < A \leq 0$, then $ \mathcal{S}^*[A,B] \subset \mathcal{G}_0$ and the semigroup generated by $f\in \mathcal{S}^*[A,B]$ satisfies
    $$ \lvert u(t,z) \rvert\leq e^{{-t (1-B)^{\frac{A-B}{B}}}} \lvert z \rvert .$$
\end{corollary}
\begin{proof}
    For $f\in \mathcal{S}^*[A,B]$,
    $$\RE \left\{ \exp\int_0^{z} \frac{{\varphi}(t)-1}{t}dt \right\} = \RE \left\{(1+B z)^{\frac{A-B}{B}} \right\}.$$ 
    Now, by taking $z=e^{i \theta}$ for $\theta\in (0,2\pi)$, we have
\begin{equation}\label{scnd}
    \RE \left\{(1+B e^{i \theta})^{\frac{A-B}{B}} \right\} = \left((1 + B \cos{\theta})^2 +B^2 \sin^2{\theta} \right)^{\frac{A-B}{2 B}} \cos{\Theta},
\end{equation}
    where
    $$ \Theta := \frac{A-B}{B} \tan^{-1}\left( \frac{\sin{\theta}}{1+ B \cos{\theta}}\right) .$$
    For $-1\leq B < A \leq 0$ and $\theta\in (0,2\pi)$,
    $$\frac{A-B}{B}\in [-1,0) \quad \text{and} \quad -\frac{\pi}{2} < \tan^{-1}\left( \frac{\sin{\theta}}{1+ B \cos{\theta}}\right) <\frac{\pi}{2}.  $$
    It is evident from the above that $\Theta \in (-\pi/2, \pi/2)$. Therefore, $\cos{\Theta} > 0$, consequently, we conclude from (\ref{scnd}) that condition (\ref{frst}) holds for the class $ \mathcal{S}^*[A,B]$ whenever $-1\leq B < A \leq 0$ and hence $ \mathcal{S}^*[A,B] \subset \mathcal{G}$.

    From (\ref{scnd}), for $-1\leq B <A \leq 0$ and $\theta\in (0,2\pi)$, we have
\begin{align*}
     \RE \left\{(1+B e^{i \theta})^{\frac{A-B}{B}} \right\} &\geq  \inf_{\theta\in (0,2\pi)} \RE (1 + B e^{i \theta})^{\frac{A-B}{B}},\\
                                                               &=(1 - B )^{\frac{A-B}{B}}.
\end{align*}
   The result now follows at once from Proposition \ref{prp1}.
\end{proof}
\begin{remark}
    Taking $A=1-2\alpha$  and $B=-1$, we see that the class $\mathcal{S}^*[A,B]$ reduces to the class  $\mathcal{S}^*(\alpha)$, $\alpha\in [0,1]$. From Corollary \ref{crl1}, we obtain $\mathcal{S}^*(\alpha)\subset \mathcal{G}_0$, whenever $\alpha\geq 1/2$ and $\lvert u(t,z) \rvert \leq 2^{-(2-2\alpha)}$, proved by Elin et al. \cite[Thereom 5]{Elin}.
\end{remark}
      For $A=0$ and $B=-1$, Corollary \ref{crl1} yields the following result:
\begin{corollary}
    $ \mathcal{S}^*(1/2) \subset \mathcal{G}_0$ and the semigroup generated by  $f\in \mathcal{S}^*(1/2)$ satisfies $\lvert u(t,z) \rvert \leq e^{-t/2} \lvert z \rvert$.
\end{corollary}
\begin{theorem}
    If $\lambda \in [0,1/3]$, then $ \mathcal{U}(\lambda) \subset \mathcal{G}_0$ and semigroup generated by $f\in \mathcal{U}(\lambda)$ satisfies
    $$ \lvert u(t,z) \rvert \leq e^{\left( \frac{t( 3 \lambda -1)}{2 \lambda^2 - 4 \lambda + 2} \right)} \lvert z \rvert.$$
    The range of $\lambda$ is best possible.
\end{theorem}
\begin{proof}
    For $f\in \mathcal{U}(\lambda)$,
\begin{equation}\label{Ulambda}
    \frac{f(z)}{z}\prec \frac{1}{(1+z)(1+\lambda z)},\;\; z\in \mathbb{D}
\end{equation}
    and
\begin{align*}
    \RE{\frac{1}{(1+z)(1+\lambda z)}} &\geq \min_{\theta\in (0,2\pi)} \RE{\frac{1}{(1+ e^{i \theta} )(1+\lambda e^{i \theta})}}\\
                                      &= \min_{\theta\in (0,2\pi)} \frac{1 - \lambda + 2 \lambda \cos{\theta}}{ 2 \lambda^2 +  4 \lambda \cos{\theta} + 2}\\
                                    &=\min_{x\in (-1,1)} \frac{ 2 \lambda x - \lambda + 1}{ 2 \lambda^2 +  4 \lambda x + 2}=: g(x),
\end{align*}
    where $x =\cos{\theta}$. Clearly,
    $g'(x)={\lambda^2 ( \lambda + 1 )}/{( \lambda^2  + 2 x \lambda + 1 )^2} >0$ for all $x\in (-1,1)$ and $\lambda\in [0,1].$
    Thus, $g(x)$ is increasing function and the infimum of its range set is attained at $x=-1$, which together with (\ref{Ulambda}) yields
    $$\RE\frac{f(z)}{z} > \frac{1 - 3 \lambda}{2 \lambda^2 - 4 \lambda + 2}$$
      and the required inclusion follows for $\lambda\in [0,1/3]$. By (\ref{classU}) and (\ref{Ulambda}), we note that
      $$ f(z) = \frac{z}{(1+z)(1+\lambda z)} \in \mathcal{U}(\lambda),$$
     and  $\RE({f(z)/z})$ may be negative when $\lambda>1/3,$ showing that the range of $\lambda$ is best possible. Now, the exponential rate of convergence of $\{u(t,z)\}_{t\geq 0}$ generated by $f\in \mathcal{U}(\lambda)$ follows by Proposition \ref{prp1}.
 \end{proof}
    The following lemma is obtained by Tuan and Anh \cite{Tuan} as a particular case of Theorem 3 of \cite{janowski} for the class of Carath\'{e}odory functions $\mathcal{P}$ (see \cite{good}).
\begin{lemma}\label{radi}
    Let $p\in \mathcal{P}$. Then for $0 \leq \alpha < 1$
\begin{align*}
    \RE \bigg( \frac{(1- \alpha) z p'(z)}{ \alpha + (1- \alpha) p(z)}\bigg) \geq
    \left\{ \begin{array}{ll}
    - \dfrac{2 (1- \alpha ) r}{(1 + (2 \alpha -1 )r )(1 + r)} & \text{for}\;\; R_1 \leq R_2, \\ \\
     - \dfrac{\alpha}{(1 - \alpha)} + \dfrac{1}{(1 - \alpha)} ( 2 R_1 - a )  & \;\; \text{for}\;\; R_2 \leq R_1
    \end{array}
    \right.
\end{align*}
    where
    $$ R_1 = \left( \frac{\alpha - \alpha (2 \alpha -1 ) r^2 }{1 - r^2} \right)^{1/2} \;\; \text{and} \;\; R_2 = \frac{1 + (2 \alpha -1 )r}{1 + r}.$$
    The results are sharp and the extremal functions are given by
\begin{enumerate}[(i)]
  \item for  $R_1 \leq R_2$, $p(z) =(1-z)/(1+z);$
  \item  for $R_2 \leq R_1$,
    $$ p(z)= \frac{1}{2} \left( \frac{1+ z e^{-i \theta}}{1 - z e^{-i \theta}} + \frac{1 + z e^{i \theta}}{1 - z e^{i \theta}} \right),$$
\end{enumerate}
    where $\cos{\theta}$ satisfies the equation
\begin{equation}\label{eqnf3}
\begin{aligned}
         (2 &\alpha -1 ) r^4 - 2 \cos{t} ( (2 R_1 - a - \alpha) (3 \alpha - 1) + (1 - \alpha)^2 ) r^3 \\
          &+  ( 2  \alpha (2 R_1 - a - \alpha) (1 + 2 \cos^2{\theta})+ 4 (1 - \alpha)^2 )r^2 \\
          &  - 2 \cos{\theta}  ( (2 R_1 - a - \alpha) (1 + \alpha) + (1 - \alpha)^2 ) r +(2 R_1  - a - \alpha)  =0
\end{aligned}
\end{equation}
    with $a= (1 - (2 \alpha - 1) r^2)/(1 - r^2).$
\end{lemma}
   For a given $r\in (0,1)$, the transition from the first case to second case takes place when $\alpha = \alpha_0 \in (0,1)$, where $\alpha_0$ is determined from the equation $R_1 = R_2.$

    Tuan and Anh \cite{Tuan} found the radius of starlikeness for the subclass of $\mathcal{S}$, which satisfies the condition $ \RE (f(z)/z)> k, 0\leq k<1$. In the following, we extend their result by considering the general class $\mathcal{S}^*(\varphi)$ in place of $\mathcal{S}^*$.
\begin{theorem}\label{ITHM}
     Let $f\in \mathcal{A}$ and $\RE (f(z)/z)> k,$ $0 \leq k < 1$. If $\inf\RE(\varphi(z)) = m$, $m\in[0,1]$, then $f\in \mathcal{S}^*(\varphi)$ in $\lvert z \rvert \leq  r_\varphi  \in (0,1)$, where $r_\varphi$ is given by
\begin{enumerate}
  \item for $k \in [0,  k_0]\setminus \{k_1 \}$,
   $$ r_{\varphi} = \frac{2 k -  m k  - 1 + \sqrt{ ( k -1 ) (k ( m -2 )^2 - (m -2 ) m -2 )}}{ (1 - 2 k) (1 - m) }, $$
   \item for $k \in  ( k_0, 1]\setminus \{k_1 \}$,
   $$ r_{\varphi} = \sqrt{ \frac{ m^2 (1 - k) -m (2 - 6 k)  - 4 k + 4 \sqrt{( m -1 ) (k -1 ) k}}{ m^2 (1 - k)  - m (4 - 8 k)  - 8 k + 4}}, $$
   \item for $k=k_1$
   $$ r_\varphi =\sqrt{\frac{m-1}{m-2}},$$
   where
   $$ k_0 = \frac{ 2 m^3 - 6 m^2  + 9 m -6 + 2 \sqrt{( m - 1)^4 (m (m - 2) + 4)}}{4 m^3 - 21 m^2 + 36 m -20} $$
   and $$ k_1 = \frac{m^2 - 4 m + 4  }{m^2 - 8 m + 8 }. $$
\end{enumerate}
   The radii is sharp.
\end{theorem}
\begin{proof}
    Since $\RE{(f(z)/z)}>k$, we can write
    $$\frac{f(z)}{z}= k+(1-k)p(z), $$
    where $p \in \mathcal{P}$. A computation shows that
\begin{equation}\label{r1r2}
     \RE \left(\frac{z f'(z)}{f(z)} \right)= \RE \left( 1+ \frac{(1-k)  z p'(z)}{k+ (1-k) p(z)}  \right).
\end{equation}
    By Lemma \ref{radi}, for $R_1 \leq R_2$, we have
\begin{align*}
    \RE \left(\frac{z f'(z)}{f(z)} \right) &\geq 1 - \frac{2 (1- k ) r}{(1 + ( 2k -1) r)(1+r)} \\
                                           &=: \xi_1(k,r).
\end{align*}
   Clearly, $f \in \mathcal{S}^*(\varphi)$ provided $\xi_1(k,r) > m$  or
\begin{align*}
    \xi_1(k,m,r):= r^2 ( 2 k - m (-1 + 2 k) -1 ) +r ( 4 k - m ( 2 k -1) - m -2)+ (1 - m) > 0.
\end{align*}
    For the case $R_1 \leq R_2$, $r_{\varphi_1}$ is the smallest positive root of $\xi_1(k,m,r)=0$ and it is given by
    $$r_{\varphi_1} =  \frac{2 k -  m k  - 1 + \sqrt{ ( k -1 ) (k ( m -2 )^2 - (m -2 ) m -2 )}}{ (1 - 2 k) (1 - m) }. $$
    Now, $r_{\varphi_1} <1$ provided
    $$ \frac{ (6 - 4 m) k -2}{(m -1) (2 k -1 )} <0, $$
    which holds when $k< 1/{(3-2 m)}.$
    Evidently, $\xi_1(k, m , 0) = 1-m \geq 0$ for all $m\in [0,1]$ and $\xi_1(k, m , 1) = -2 + (6 - 4 m) k < 0$ for $k < 1/{(3-2 m)}$, which ensures the existence of root $r_{\varphi_1} \in (0,1)$.

    For the case $R_2 \leq R_1$, $r_{\varphi_2}$ is the smallest positive root of $\xi_2(k,m, r)=0$, where
\begin{equation}\label{sigma2}
   \xi_2(k,m, r) = 1 - \frac{k}{(1 - k)} + \frac{1}{(1 - k)} \left( 2 \left( \frac{k - k (2 k - 1) r^2}{1 - r^2} \right)^{1/2} - \frac{1 - (2 k - 1) r^2}{(1 - r^2)} \right) - m
\end{equation}
    and the root is
    $$ r_{\varphi_2} = \sqrt{ \frac{ m^2 (1 - k) -m (2 - 6 k)  - 4 k + 4 \sqrt{( m -1 ) (k -1 ) k}}{ m^2 (1 - k)  - m (4 - 8 k)  - 8 k + 4}},$$
  where $k \neq (4 - 4 m + m^2)/(8 - 8 m + m^2)$.
  For $k= (4 - 4 m + m^2)/(8 - 8 m + m^2)$, (\ref{sigma2}) yields the root
  $$ r_{\varphi_3} = \sqrt{\frac{m-1}{m-2}}.  $$
  For fixed $m \in [0,1]$, $r_{\varphi_2} \in (0,1)$ whenever $k > m^2/(m-2)^2.$
  Clearly
  $$r_{\varphi_1} =r_{\varphi_2}$$
  when
  $$k_0 = \frac{ 2 m^3 - 6 m^2  + 9 m -6 + 2 \sqrt{( m - 1)^4 (m (m - 2) + 4)}}{4 m^3 - 21 m^2 + 36 m -20}.$$
    For $m\in [0,1]$, $k_0$ lies in $[1/10,1].$

\noindent{\bf{Sharpness:}} For $k \in [0,k_0] \setminus \{ k_1 \}$, sharpness follows for
   $$ f(z)= z \left(\frac{ 1+ (2 k - 1) z }{1 + z} \right) $$
   as $ z f'(z)/f(z) = m$ when $z=r_{\varphi_1}.$
   For second inequality, extremal function is given by
   $$ f(z)= k z+ \frac{(1- k) }{2} z \left( \frac{1 + z e^{-i \theta}}{1 - z e^{-i \theta}} + \frac{1 + z e^{i \theta}}{1 - z e^{i \theta}} \right),$$
   where $\cos{\theta}$ satisfy (\ref{eqnf3}) with $r= r_{\varphi_2}$.
\end{proof}
   For $f\in \mathcal{A}_\beta$, $\beta \in [0,1]$, Bracci et al. \cite{FB1} proved that
   $$\RE \frac{f(z)}{z} \geq \kappa (\beta) = \int_0^1 \frac{1-t^{1-\beta}}{1+ t^{1-\beta}} dt>0 , \quad z\in \mathbb{D},$$
  where $\kappa(\beta)$ is a decreasing function and it maps $[0,1]$ onto $[2 \ln{2}-1,0]$. Thus, we easily obtain the following result for the class $\mathcal{A}_\beta$ from Theorem \ref{ITHM}.
\begin{theorem}\label{thmg}
    If $f\in \mathcal{A}_\beta$, $0 \leq \beta \leq 1$, then $f\in \mathcal{S}^*(\varphi)$ in $ \mathbb{D}_{r_{\varphi}}=\left\{z\in \mathbb{C}: \lvert z \rvert \leq r_\varphi   \right\}$, where
\begin{enumerate}
   \item for $\beta \in  (0, \beta^*]$,
   $$ r_{\varphi} = \sqrt{ \frac{ m^2 (1 - \kappa(\beta)) -m (2 - 6 \kappa(\beta))  - 4 \kappa(\beta) + 4 \sqrt{( m -1 ) (\kappa(\beta) -1 ) \kappa(\beta)}}{ m^2 (1 - \kappa(\beta))  - m (4 - 8 \kappa(\beta))  - 8 \kappa(\beta) + 4}}, $$
  \item for $\beta \in [\beta^*, 1]$,
   $$ r_{\varphi} =  \frac{2 \kappa(\beta) -  m \kappa(\beta)  - 1 + \sqrt{ ( \kappa(\beta) -1 ) (\kappa(\beta) ( m -2 )^2 - (m -2 ) m -2 )}}{ (1 - 2 \kappa(\beta)) (1 - m) }, $$
\end{enumerate}
   and  $\beta^*$ is the root of
   $$ \int_0^1 \frac{1- t^{1-\beta^*}}{1+t^{1-\beta^*}} dt= \frac{ 2 m^3 - 6 m^2  + 9 m -6 + 2 \sqrt{( m - 1)^4 (m (m - 2) + 4)}}{4 m^3 - 21 m^2 + 36 m -20}.$$
\end{theorem}
    By taking $\varphi(z)=(1+A z)/(1+B z)$ in Theorem \ref{thmg}, we have $m=(1-A)/(1-B)$ and the following result follows:
\begin{corollary}\label{crl2}
    If $f\in \mathcal{A}_\beta$, then $f\in \mathcal{S}^*[A,B]$ in $\lvert z \rvert < r$, where
\begin{enumerate}[(i)]
  \item for $\beta \in  [0, \beta^*]$,
\begin{align*}
    r = &\bigg( \big( 4 \sqrt{(A - B) (1 - B)^3 (1 - \kappa(\beta)) \kappa(\beta)} - (1 - A) (1 + A - 2 B) \\
        &+ (1 - A (4 + A) + 2 B + 6 A B - 4 B^2) \kappa(\beta) \big)\bigg/\big((1 + A -  2 B)^2 \\
      &- (1 + A^2 + A (6 - 8 B) - 8 (1 - B) B) \kappa(\beta)\big) \bigg)^{1/2},
\end{align*}
 \item for $\beta \in [\beta^* , 1]$,
\begin{align*}
   r = & \bigg(\sqrt{(1 - \kappa(\beta)) (1 +
     A^2 - 2 B - 2 A B + 2 B^2 - (1 + A - 2 B)^2 \kappa(\beta))} \\
     & + (1 - B) - ( 1 + A - 2 B) \kappa(\beta) \bigg)\bigg/(2 (A - B) (-1 + 2 \kappa(\beta))),
\end{align*}
\end{enumerate}
  and $\beta^*$ is the root of
\begin{align*}
     \int_0^1 \frac{1- t^{1-\beta^*}}{1+t^{1-\beta^*}} dt= &(B -1)^3 \bigg/\bigg( 2 \sqrt{(A - B)^4 (3 + A^2 - 2 (3 + A) B + 4 B^2)}  \\
         & -1 - 3 A - 2 A^3 + 6 (1 + A + A^2) B - 9 (1 + A) B^2 + 6 B^3 \bigg).
\end{align*}
\end{corollary}
    Corollary \ref{crl2} yields the following result for $A=(1-2 \alpha)$ and $B=-1.$
\begin{corollary}
   If $f\in \mathcal{A}_\beta$, then $f\in \mathcal{S}^*(\alpha)$ in $\lvert z \rvert =r <1 $, where
\begin{enumerate}[(i)]
  \item for $\beta \in  (0, \beta^*]$,
\begin{align*}
    r = \sqrt{\frac{  \kappa(\beta) (  \alpha^2 - 6 \alpha + 4) - \alpha^2 + 2 \alpha -4 \sqrt{\kappa(\beta) (1 - \kappa(\beta)) (1 - \alpha)}}{
 k (\alpha^2 - 8 \alpha + 8) -(\alpha -2 )^2}},
\end{align*}
 \item for $\beta \in [\beta^*, 1]$,
\begin{align*}
   r = \frac{2 \kappa(\beta) (2 - \alpha) - 2 - \sqrt{ (
    \kappa(\beta) (\alpha - 2)^2 + \alpha (\alpha - 2) + 2 )(1 - \kappa(\beta))}}{2 (1 -
   2 \kappa(\beta)) (1 - \alpha)},
\end{align*}
\end{enumerate}
   and $\beta^*$ is the root of
\begin{align*}
     \int_0^1 \frac{1- t^{1-\beta^*}}{1+t^{1-\beta^*}} dt= \frac{1}{6 - 9 \alpha + 6 \alpha^2 - 2 \alpha^3 - 2 \sqrt{( \alpha -1)^4 ( \alpha^2 - 2 \alpha + 4)}}.
\end{align*}
\end{corollary}
\begin{remark}
    For $\varphi(z)=(1+z)/(1-z)$, Theorem \ref{thmg} gives the radii of starlikeness for $f\in \mathcal{A}_\beta$ \cite[Theorem 8]{Elin}.
\end{remark}
    For various choices of $\varphi(z)$ such as $2/(1+e^{-z})$, $1+ z e^z$ and $1+ (2\pi^2) (\log(1+\sqrt{z})/(1-\sqrt{z}))^2$, the class $\mathcal{S}^*(\varphi)$ respectively reduces to the subclasses of starlike functions $\mathcal{S}^*_{SG}$, $\mathcal{S}^*_\varrho$ and $\mathcal{S}_{P}$  (see \cite{goel,kamal,ronn}). The following corollary gives the radii for these subclasses.
\begin{corollary}
\begin{enumerate}[(i)]
  \item If $f\in \mathcal{A}_\beta$, then $f\in \mathcal{S}^*_{SG}$ in $\lvert z \rvert = r <1$, where
\begin{enumerate}
   \item for $\beta \in  (0, \beta^*]$,
   $$ r =\sqrt{\frac{e \kappa(\beta)+ \kappa(\beta) -e - e^2 \kappa(\beta) + (1 + e) \sqrt{(e^2 -1)(1 - \kappa(\beta)) \kappa(\beta)}}{e^2 (1 - 2 \kappa(\beta)) + \kappa(\beta)}}, $$
   \item for $\beta \in [\beta^*, 1]$,
   $$ r =  \frac{1 + e - 2 e \kappa(\beta) - \sqrt{2 (1 - \kappa(\beta)) (1 - e^2 (2 \kappa(\beta) -1 ))}}{(1 - e) (1 - 2 \kappa(\beta))}, $$
\end{enumerate}
   and $\beta^*$ is the root of
   $$ \int_0^1 \frac{1- t^{1-\beta^*}}{1+t^{1-\beta^*}} dt= \frac{ 2 ( 1 - 2 e  +  e^2 ) \sqrt{1 + e + e^2} - 3 e^3 - 3 e + 2 }{6 e^2 - 10 e^3} .$$
  \item If $f\in \mathcal{A}_\beta$, then $f\in \mathcal{S}_P$ in $\lvert z \rvert <r$, where
\begin{enumerate}
   \item for $\beta \in  (0, \beta^*]$,
   $$ r =  \sqrt{\frac{ 5 \kappa(\beta) + 3 - 8 \sqrt{2} \sqrt{(1 - \kappa(\beta)) \kappa(\beta)}}{17 \kappa(\beta) - 9}} , $$
   \item for $\beta \in [\beta^*, 1]$,
   $$ r =  \frac{3 k -2 + \sqrt{9 \kappa(\beta)^2 - 14 \kappa(\beta) + 5 }}{1 - 2 k} , $$
   and $\beta^*$ is the root of
    $$ \int_0^1 \frac{1- t^{1-\beta^*}}{1+t^{1-\beta^*}} dt = \frac{11 - \sqrt{13}}{27} .$$
\end{enumerate}
\end{enumerate}
\end{corollary}
    For $\beta = 1$, $\kappa(\beta)=0$ and Theorem \ref{thmg} yields the following result for the class $\mathcal{G}.$
\begin{corollary}
\begin{enumerate}[(i)]
  \item If $f\in \mathcal{G}_0$, then $f\in \mathcal{S}^*_{SG}$ on the disk of radius $ r \approx 0.219887  $.
  \item If $f\in \mathcal{G}_0$, then $f\in \mathcal{S}_{P}$ on the disk of radius $ r=\sqrt{5} -2  $.
  \item  If $f\in \mathcal{G}_0$, then $f\in \mathcal{S}^*_{\varrho}$ on the disk of radius $ r \approx 0.372153.   $
\end{enumerate}
\end{corollary}

\section{Declarations}
\subsection*{Funding}
The work of the Surya Giri is supported by University Grant Commission, New Delhi, India,  under UGC-Ref. No. 1112/(CSIR-UGC NET JUNE 2019).
\subsection*{Conflict of interest}
	The authors declare that they have no conflict of interest.
\subsection*{Author Contribution}
    Each author contributed equally to the research and preparation of manuscript.
\subsection*{Data Availability} Not Applicable.
\noindent

\end{document}